\documentclass[11pt]{amsart}

\usepackage{amssymb}
\usepackage{latexsym}
\usepackage{amsmath}
\usepackage{euscript}

      \def\dC{{\mathbb C}}

   \def\dN{{\mathbb N}}   
      \def\dR{{\mathbb R}}

\newcommand\gt{{\mathfrak{t}}}

\def\bm\chi{\mbox{\boldmath$\chi$}}

\def\dom{{\rm dom\,}}

\def\rank{{\rm rank\,}}

\def\e{{\rm e}}
\def\I{{\rm i}}
\let\xker=\ker \def\ker{{\xker\,}}

\def\sgn{{\rm sgn\,}}

\newcommand{\mB}{B}
\newcommand{\mC}{C}
\newcommand{\mD}{D}

\newtheorem{theorem}{Theorem}[section]
\newtheorem{proposition}[theorem]{Proposition}
\newtheorem{corollary}[theorem]{Corollary}
\newtheorem{lemma}[theorem]{Lemma}

\theoremstyle{definition}
\newtheorem{example}[theorem]{Example}
\newtheorem{definition}[theorem]{Definition}
\newtheorem{remark}[theorem]{Remark}

\numberwithin{equation}{section}

\begin{document}
\title[The Riesz basis property of an indefinite Sturm-Liouville problem]
{The Riesz basis property of\\ an indefinite Sturm-Liouville problem\\
with non-separated boundary conditions}
\author[B.~\'Curgus]{Branko \'{C}urgus}
\address{Department of Mathematics\\
Western Washington University\\  Bellingham WA\\
98225, USA}
\email{curgus@gmail.com}

\author[A.~Fleige]{Andreas Fleige}
\address{Baroper Schulstrasse 27 A \\
44225 Dortmund \\
Deutschland}
\email{andreas-fleige@versanet.de}

\author[A.~Kostenko]{Aleksey Kostenko}
\address{Fakult\"at f\"ur Mathematik\\
Universit\"at Wien\\
Oskar-Morgenstern-Platz 1\\ 1090 Wien\\ Austria}
\email{duzer80@gmail.com}

\thanks{ Integr. Equat. Oper. Theory (2013), to appear}

\thanks{{\it The research was funded by the Austrian Science Fund (FWF): M1309}}

\subjclass{Primary 34B09; Secondary 34B24, 34L10, 47B50, 26D10}

\keywords{indefinite Sturm-Liouville problem, Riesz basis, regular critical point, HELP inequality}

\begin{abstract}
 We consider a regular indefinite Sturm-Liouville eigenvalue problem
\mbox{$-f'' + q f = \lambda r f$} on $[a,b]$
subject to general self-adjoint boundary conditions and
with a weight function $r$ which changes its sign at finitely many, so-called turning points.  We give sufficient and in some cases necessary and sufficient conditions for the Riesz basis property of this eigenvalue problem.
In the case of separated boundary conditions we extend the class of weight functions $r$ for which the Riesz basis property can be completely  characterized in terms of the local behavior of $r$ in a neighborhood of the turning points. We identify a class of  non-separated boundary conditions for which, in addition to the local behavior of $r$ in a neighborhood of the turning points, local conditions on $r$ near the boundary are needed for the Riesz basis property. As an application, it is shown that the Riesz basis property for the periodic boundary conditions is closely related to a regular HELP-type inequality without boundary conditions.
\end{abstract}

\maketitle


\section{Introduction}
Let $q$ and $r$ be real integrable functions on $[a,b]$, $-\infty < a < b < \infty$.  We assume that $r \neq 0$ a.e. on $[a,b]$ and $r$ changes sign at finitely many, say $n \geq 1$, points in $(a,b)$.  To be more precise, we assume that there exists a polynomial $p$ of degree $n$ whose roots are simple and lie in $(a,b)$ such that $rp > 0$ a.e. on $[a,b]$.  Then the roots of $p$ are the points where $r$ changes sign.  These points are also known as turning points of $r$.  We consider the regular indefinite Sturm-Liouville eigenvalue problem
\begin{equation}\label{eq:sp}
-f''+qf=\lambda r f\quad \text{on} \quad [a,b]
\end{equation}
subject to the self-adjoint boundary conditions
\begin{equation}\label{eq:gbc}
\mC \begin{pmatrix}
 f'(a)\\
 -f'(b)
\end{pmatrix}
 =
 \mD\begin{pmatrix}
 f(a)\\
 f(b)
\end{pmatrix}.
\end{equation}
Here $\mC, \mD\in \dC^{2\times 2}$ satisfy
\begin{equation}\label{CD}
\rank(\mC|\mD)=2,\qquad \mC\mD^*=\mD\mC^*.
\end{equation}
In particular, we study the following problems:
\begin{equation}\label{separated}
-f'' + q f = {\lambda} r f
\quad \mbox{on} \quad [a,b], \qquad
f(a) = f(b) = 0,
\end{equation}
i.e.,  $\mD=I_2$ and $ \mC=0$ in \eqref{eq:gbc};
\begin{equation}\label{nonseparated}
-f'' + q f = {\lambda} r f
\quad \mbox{on} \quad [a,b], \quad
\e^{\I t} f(a) = f(b), \quad f'(a) =  \e^{-\I t}  f'(b) + d \, f(a)
\end{equation}
i.e., $\mC=\begin{pmatrix}
 1 & \e^{-\I t} \\
 0 & 0
 \end{pmatrix}
 $ and
 $\mD=\left(\!\!\begin{array}{cc}
 d & 0\\
 -\e^{\I t} & 1
 \end{array}\!\!\right)$ with $t \in [0, 2\pi)$ and $d \in \dR$.

At the beginning of Section~\ref{reduction}, with eigenvalue problem  \eqref{eq:sp}, \eqref{eq:gbc} with \eqref{CD} we associate the operator $L_{C,D}$ in the Hilbert space $L^2_{|r|}[a,b]$.  The spectral properties of $L_{C,D}$ are identical to the spectral properties of the corresponding  eigenvalue problem. Since we assume that $r$ is indefinite, the operator   $L_{C,D}$ is not self-adjoint in the Hilbert space $L^2_{|r|}[a,b]$. However, the space $L^2_{|r|}[a,b]$ equipped with the indefinite inner product $[f,g] = \int_a^b f \overline{g}\,  r dx $, $f,g \in L^2_{|r|}[a,b]$,  becomes a Krein space. It follows from results in \cite{CL} that the operator $L_{C,D}$ is definitizable in this Krein space, its spectrum is discrete with at most finitely many non-semisimple eigenvalues and the linear span of the root functions (i.e. eigenfunctions and associated functions) of $L_{C,D}$ is dense in $L^2_{|r|}[a,b]$.

One of the main goals of this paper is to provide conditions on the coefficients $q$, $r$, $\mC$ and $\mD$ under which there is a Riesz basis of the Hilbert space $L_{|r|}^2[a,b]$ which consists of root functions of eigenvalue problem  \eqref{eq:sp}, \eqref{eq:gbc} subject to \eqref{CD}, that is of root functions of the operator $L_{C,D}$.
This will be referred to as the {\em Riesz
basis property} of problem \eqref{eq:sp}, \eqref{eq:gbc} in $L^2_{|r|}[a,b]$.

The Riesz basis property of problem \eqref{separated} (i.e. with Dirichlet boundary conditions) and with one turning point has been extensively investigated during the last three decades (see for example \cite{AP, B, BF, BF2, CN, F2, F3, F5, KKM_09, KM_08, K, K2, P, Pa2, sP2, V}). 
In \cite{B} Beals gave the first  general sufficient condition for the Riesz basis property. This condition  requires the weight function to behave like a power of the independent variable $x$ in a neighborhood of the turning point. Using Pyatkov's interpolation criterion \cite{sP}, Parfenov \cite{P} (see also \cite{Pa2}) discovered conditions on the local behavior of $r$ at the turning point which turned out to be necessary and sufficient for the Riesz basis property if the weight function $r$ is odd (see Theorem~\ref{th:parfenov} below).  Another proof of Parfenov's criterion was given recently in \cite{K3}. A generalization to certain classes of non-odd weights was obtained in,  e.g., \cite{F5}, \cite{BF}, \cite{BF2}. In the present paper this generalization will be improved again extending Parfenov's equivalence statement to odd-dominated weight functions $r$, see Theorem~\ref{lem:odddom}. Furthermore, we point out a  connection of Parfenov's conditions to the class of so-called positively increasing functions (see, e.g., \cite{BGT, bks, Rog} and also \cite{K2, K3}).

The main focus of this paper is on extending of the results from the previous paragraph to general boundary conditions.  The example in \cite{BC2} shows that in the case of antiperiodic boundary conditions imposing only a condition at the local behavior of $r$ at the turning point is not sufficient to guarantee the Riesz basis property of eigenvalue problem \eqref{eq:sp}.  Some results involving general self-adjoint boundary conditions and multiple turning points at which $r$ satisfies a Beals-type condition were obtained in \cite{CL}.  In \cite{sP2} Pyatkov obtained several results related to self-adjoint boundary conditions without any assumptions on the weight function at the turning points. For example, provided that the number of turning points is even (i.e., $r$ is of one sign in a neighborhoods of $a$ and $b$), it was shown that if the Riesz basis property holds for one set of self-adjoint boundary conditions, then it holds for all self-adjoint boundary conditions.  Further, for an arbitrary finite number of turning points, it was shown that if the Riesz basis property holds for one set of separated boundary conditions, then it holds for all separated boundary conditions.  In the case of non-separated boundary conditions and an odd number of turning points, in addition to local conditions at the turning points, some local conditions on $r$ near the boundary points are needed for the Riesz basis property to hold. Such sufficient conditions given in \cite{sP2} are more general then those in \cite{CL}.  The relevant results from \cite{sP2} are reviewed in Section~\ref{sec:pyatkov}.  We note that the methods of \cite{CL} were used in \cite{BC} to give sufficient conditions for the Riesz basis property of problem \eqref{eq:sp} subject to linear self-adjoint $\lambda$-dependent boundary conditions.

For problem \eqref{eq:sp}, \eqref{eq:gbc} with $C,D$ satisfying \eqref{CD} and with the weight function $r$ which is locally odd at the boundary (and hence has an odd number of turning points) we improve the existing results in several ways. First, in Subsection~\ref{ss:4.2} we prove that the Riesz basis property of \eqref{eq:sp}, \eqref{eq:gbc} depends only on the local behavior of $r$ in neighborhoods of the turning points if and only if the boundary conditions \eqref{eq:gbc} are not row equivalent to the boundary conditions in \eqref{nonseparated}.   In other words, additional conditions on $r$ in a neighborhood of the boundary are needed for the Riesz basis property of \eqref{eq:sp}, \eqref{eq:gbc} if and only if \eqref{eq:gbc} are  row equivalent to \eqref{nonseparated}. Since the boundary conditions in \eqref{nonseparated} include the antiperiodic boundary conditions this result provides an abstract framework for the example in \cite{BC2}. Further assuming that $r$ is locally odd-dominated at all turning points in Theorem~\ref{ThnMain} we give a necessary and sufficient conditions for the Riesz basis property of the general problem (\ref{eq:sp}), (\ref{eq:gbc}).

In Section~5 we apply our result for problem \eqref{nonseparated} to the regular HELP-type inequality
\begin{equation} \label{Everitt}
\Big(\int_{0}^b \frac{1}{r} \,|f'|^2 + q |f|^2 \; dx \Big)^2 \leq K \Big(\int_{0}^b |f|^2\, dx\Big) \Big(
\int_{0}^b \Big|-\Big(\frac{1}{r} f'\Big)' + q f\Big|^2  dx\Big)  
\end{equation}
where $b > 0$, $q, r \in L^1[0,b]$ and $r > 0$ a.e. on $(0,b)$. In fact, with the odd extension of $r$ to $[-b,b]$ and an arbitrary extension of $q \in L^1[-b,b]$ a close relation between the Riesz basis property of \eqref{separated} and the validity of inequality \eqref{Everitt} for all suitable functions $f$ with $(f'/r)(b)=0$ was established in \cite{V, BF2, K2, K3}.  The restriction $(f'/r)(b)=0$ originates from the approach to inequality \eqref{Everitt} used in \cite{Et, EE1, EE2}.
However, in  \cite{Bz, Bz1} Bennewitz studied inequality \eqref{Everitt}
without this restriction.  In Subsection~\ref{sec:bennezitz} we show that the validity of inequality \eqref{Everitt}
in Bennewitz' sense is equivalent to the Riesz basis property of \eqref{nonseparated} with an additional algebraic condition on the solutions $f$ of $-(\frac{1}{r} f')' + q f=0$ which goes back to Bennewitz.

\section{A reduction of the general problem to some specific problems}\label{reduction}

\subsection{The operator associated with the boundary value problem}
With problem \eqref{eq:sp}, \eqref{eq:gbc} we associate the operator $L_{\mC,\mD}$ defined in $L^2_{|r|}[a,b]$ by the differential expression
\begin{equation*} 
\ell[f] = \frac{1}{r}\left(-f''+qf\right)
\end{equation*}
on the domain
\begin{equation*} 
 \dom(L_{\mC,\mD})=\big\{f\in L^2_{|r|}[a,b]: f, f'\in AC[a,b], \,  
 \ell[f]\in L^2_{|r|},  f  \
  \text{satisfies \eqref{eq:gbc}} \big\}.
\end{equation*}
The following two inner products on $L^2_{|r|}[a,b]$ arise naturally in this setting:
\[
[f,g] := \int_a^b f \bar{g} \ r \; dx, \qquad(f,g) := \int_a^b f \bar{g}\ |r| \; dx.
\] 
The space $\bigl(L^2_{|r|}[a,b],[\cdot,\cdot]\bigr)$ is a Krein space,  $J f := \sgn(r) f, f \in L^2_{|r|}[a,b],$ is a fundamental symmetry on $L^2_{|r|}[a,b]$ and $(\cdot,\cdot) = [J\cdot,\cdot]$ is the associated Hilbert space inner product.  If $C$ and $D$ satisfy \eqref{CD}, it was shown in \cite{CL} that the operator $L_{\mC,\mD}$ 
is definitizable in $(L^2_{|r|}[a,b], [\cdot,\cdot])$, its spectrum is discrete and infinity is its critical point. For the theory of definitizable operators and their critical points we refer to \cite{L}. The following result from \cite{CN} is useful in this context.

\begin{proposition}[\cite{CN}]\label{ThmCurgusNaiman}
The problem \eqref{eq:sp}, \eqref{eq:gbc}, \eqref{CD} has the  Riesz basis property  if and only if infinity is a regular critical point of $L_{\mC,\mD}$.
\end{proposition}

\subsection{A classification of the boundary conditions}\label{operators}
We recall that throughout the paper we assume that $\mC,\mD\in \dC^{2\times 2}$ satisfy \eqref{CD}.  Notice that if
$\mC_1 = M \mC$ and $\mD_1 = M\mD$ with an invertible $M\in \dC^{2\times 2}$, then $\mC_1,\mD_1\in \dC^{2\times 2}$ satisfy \eqref{CD} and  $L_{\mC,\mD} = L_{\mC_1, \mD_1}$. Therefore, we distinguish three main cases depending on $\rank \mC \in \{0, 1, 2\}$.

\subsubsection{The case $\rank \mC=2$:}\label{sub:rank2}
In this case one can choose $\mC$ and $\mD$ such that
\[
\mC=I_2,\quad \mD = \mB = \mB^* = \left(\!\!\begin{array}{cc}
 b_{11} & b_{12}\\
 \overline{b}_{12} & b_{22}
 \end{array}\!\!\right).
\]
Then, \eqref{eq:gbc} has the following form
\begin{equation}\label{eq:c1}
\begin{cases}
f'(a) = b_{11}f(a)+b_{12}f(b)\\
-f'(b) = \overline{b}_{12}f(a)+b_{22}f(b)
\end{cases}.
\end{equation}
In this case we set $L_{\mB} := L_{I,\mB}$,
i.e. $L_{\mB} = L_{\mC,\mD}$ with $\mC=I_2,\; \mD=\mB=\mB^*$.
Note that the boundary conditions (\ref{eq:c1}) are separated precisely when $\mB$ is diagonal, i.e., $b_{12}=0$.

\subsubsection{The case $\rank \mC = 0$:}
Since we assume that $\rank(\mC|\mD)=2$,  we can set
\[
\mC=0,\quad \mD=I_2.
\]
That is, this case corresponds to the Dirichlet boundary conditions
\begin{equation}\label{eq:c3}
f(a)=f(b)=0
\end{equation}
and we denote the corresponding operator by $L^F := L_{\mC,\mD}$.

\subsubsection{The case $\rank \mC=1$:}
Without loss of generality we can assume that the matrix $\mC$
admits the following representation
\[
\mC=\left(\!\!\begin{array}{cc}
 c_{1} & c_{2}\\
 0 & 0
 \end{array}\!\!\right),\quad |c_1|+|c_2|\neq 0.
\]
Further, the symmetry condition (\ref{CD}) can be rewritten as follows:
\[
\rank(\mC|\mD)=2,\quad
\begin{cases}
c_1\overline{d}_{11}+c_2\overline{d}_{12}\in \dR\\
c_1\overline{d}_{21}+c_2\overline{d}_{22}=0
\end{cases}.
\]
Again there are three possibilities:
\begin{enumerate}
\item[a)] $c_1=0$, \ $c_2\neq 0$,
\item[b)] $c_1\neq0$, \ $c_2= 0$,
\item[c)] $c_1c_2\neq 0$.
\end{enumerate}
Consider these cases in detail.
\begin{enumerate}
\item[a)] If $c_1=0$, then we can choose
$\mC=\left(\!\!\begin{array}{cc}
 0 & 1\\
 0 & 0
 \end{array}\!\!\right)$ and hence the symmetry conditions (\ref{CD}) yield that the matrix $\mD$ admits the following representation
 \[
 \mD=\left(\!\!\begin{array}{cc}
 0 & d\\
 1 & 0
 \end{array}\!\!\right),\quad d\in \dR.
 \]
Therefore, \eqref{eq:gbc} takes the following form
\begin{equation}\label{eq:c2.1}
f(a)=0,\qquad f'(b)+df(b)=0,\quad d\in\dR.
\end{equation}
In this case we denote the associated operator by $L^b_d := L_{\mC,\mD}$.
\item[b)] Similarly, if $c_2=0$, then $\mC$ and $\mD$ are as follows
\[
\mC=\left(\!\!\begin{array}{cc}
 1 & 0\\
 0 & 0
 \end{array}\!\!\right),\qquad
 \mD=\left(\!\!\begin{array}{cc}
 d & 0\\
 0 & 1
 \end{array}\!\!\right),\quad d\in \dR,
\]
and the boundary conditions \eqref{eq:gbc} become
\begin{equation}\label{eq:c2.2}
f'(a)-df(a)=0, \qquad f(b)=0,\quad d\in\dR.
\end{equation}
In this case we denote the associated operator by $L^a_d := L_{\mC,\mD}$.

\item[c)] Finally, let $c_1c_2\neq 0$. Then we can set
$\mC=\left(\!\!\begin{array}{cc}
 1 & \overline{c}\\
 0 & 0
 \end{array}\!\!\right)$, where $c\neq 0$. Substituting $\mC$ into the second condition in (\ref{CD}), we arrive at the following conditions
 \[
\rank(\mC|\mD)=2,\quad
\begin{cases}
d_{11} + c d_{12}\in \dR \\
d_{21} + c d_{22} = 0
\end{cases}.
 \]
The conditions $\rank(\mC|\mD)=2$ and $d_{21} + c d_{22} = 0$ imply $d_{21}d_{22}\neq 0$. Therefore, we can choose $d_{22}=1$ and $d_{21}=-c$. Finally, we arrive at the following representation
 \begin{equation}\label{eq:c2.3matr}
 \mC=\left(\!\!\begin{array}{cc}
 1 & \overline{c}\\
 0 & 0
 \end{array}\!\!\right),\quad
 \mD=\left(\!\!\begin{array}{cc}
 d & 0\\
 -c & 1
 \end{array}\!\!\right),\quad c\in\dC\setminus\{0\},\ d\in\dR,
 \end{equation}
 and hence \eqref{eq:gbc} takes the form
 \begin{equation}\label{eq:c2.3}
\begin{cases}
c f(a)=f(b) \\
f'(a)-\overline{c}f'(b)=df(a)
\end{cases},\quad d\in\dR,\ c\neq 0.
\end{equation}
The operator associated with the boundary conditions \eqref{eq:c2.3} will be denoted by $L_{c,d}$, i.e., $L_{c,d}:=L_{\mC,\mD}$ when $\mC,\mD$ have the form \eqref{eq:c2.3matr}.
Note that the boundary conditions in \eqref{nonseparated}
are of the form \eqref{eq:c2.3} with $c=e^{it}$.
\end{enumerate}

\subsection{Form domains}
In this subsection we use some known results on form domains
to get further simplifications of the boundary conditions and of the coefficients in the differential expression.  The operator $L_{\mC,\mD}$ induces the quadratic form
\[
\gt_{\mC,\mD}^0[f] := [L_{\mC,\mD} f,f], \qquad f \in \dom(\gt_{\mC,\mD}^0):=\dom(L_{\mC,\mD})
\]
in $L^2_{|r|}[a,b]$.
The form $\gt_{\mC,\mD}^0$ is bounded from below and closable
in the Hilbert space $L^2_{|r|}[a,b]$ (see, e.g. \cite[\S 13]{W}). Denote by $\gt_{\mC,\mD}$ its closure in $L^2_{|r|}[a,b]$, $\gt_{\mC,\mD}=\overline{\gt_{\mC,\mD}^0}$. For the particular cases of $L^F$, $L_{c,d}$ and $L_\mB$
this form is denoted by $\gt^F$, $\gt_{c,d}$, and $\gt_\mB$,
respectively.
The following  description of the domains of
$\gt^F$, $\gt_{c,d}$ and $\gt_\mB$
was obtained by M.G. Krein \cite[\S 6]{kre46} in the case $r\equiv 1$.
However, the methods used there extend to the case of arbitrary weights $r\in L^1[a,b]$
(see also \cite[Section 2.2]{CL}).

\begin{proposition}
\label{ThmFormDomain}
The form domain of $\gt_{\mC,\mD}$ does not depend on $q \in L^1[a,b]$.
Moreover, the form domains of
$\gt^F$, $\gt_{c,d}$ and $\gt_\mB$ are given by
\begin{equation} \notag 
 \dom(\gt^F)=\big\{f\in L^2_{|r|}[a,b]: \, f\in AC[a,b], \int_a^b|f'|^2dx<\infty,\ f(a)=f(b)=0\big\},
\end{equation}
and, for $c\neq 0, d\in\dR$,
\begin{equation} \notag 
 \dom\big(\gt_{c,d}\big)=\big\{f\in L^2_{|r|}[a,b]: \, f\in AC[a,b], \int_a^b|f'|^2dx<\infty,\ c f(a)=f(b)\big\},
\end{equation}
and, for $\mB=\mB^*$,
\begin{equation} \notag 
 \dom\big(\gt_\mB\big)=\big\{f\in L^2_{|r|}[a,b]: \ f\in AC[a,b],\ \int_a^b|f'|^2dx<\infty\big\}.
\end{equation}
\end{proposition}

The following proposition is an immediate consequence of \cite[Corollary~3.6]{C}.

\begin{proposition}\label{ThmCurgusForm}
Let $q, r \in L^1[a,b]$ and assume that both pairs $\mC,\mD \in \dC^{2\times 2}$ and $\tilde{\mC},\tilde{\mD} \in \dC^{2\times 2}$ satisfy  \eqref{CD}.  Further assume $\dom\bigl(\gt_{\mC,\mD}\bigr) = \dom\bigl(\gt_{\tilde{\mC},\tilde{\mD}}\bigr)$.  Then infinity is a regular critical point of $L_{\mC,\mD}$ if and only if infinity is a regular critical point of $L_{\tilde{\mC},\tilde{\mD}}$.
\end{proposition}

By Proposition \ref{ThmFormDomain} the closed form domains do not depend either on the choice of the potential
$q \in L^1[a,b]$ or on $\mB=\mB^*$ and $d\in\dR$.
Therefore, Proposition \ref{ThmCurgusForm} together with Proposition \ref{ThmCurgusNaiman} immediately imply

\begin{corollary}\label{CorCurgus}
Let $q, r \in L^1[a,b]$.
\begin{enumerate}
\item[(i)] If problem \eqref{eq:sp}, \eqref{eq:gbc} has the Riesz basis property for one $q \in L^1[a,b]$, then it has the Riesz basis property for all $q \in L^1[a,b]$.
\item[(ii)] If problem \eqref{eq:sp}, \eqref{eq:c1} has the Riesz basis property for one matrix $\mB=\mB^*\in \dC^{2\times 2}$, then it has the Riesz basis property for all $\mB=\mB^*\in \dC^{2\times 2}$.
\item[(iii)] If problem \eqref{eq:sp}, \eqref{eq:c2.3} has the Riesz basis property for one $d\in\dR$, then it has the Riesz basis property for all $d\in\dR$.
\end{enumerate}
\end{corollary}


\subsection{Conclusion}\label{sub:conclusion}
In order to discuss the Riesz basis property of (\ref{eq:sp}), (\ref{eq:gbc}), i.e. of the operator $L_{\mC,\mD}$,
it is enough to treat the eigenvalue problem (\ref{eq:sp}) together
with the boundary conditions of the types
(\ref{eq:c1}), (\ref{eq:c3}), (\ref{eq:c2.1}), (\ref{eq:c2.2}) and (\ref{eq:c2.3}),
i.e. the operators $L_\mB, L^F, L^b_d, L^a_d$ and $L_{c,d}$. Moreover,
by Corollary \ref{CorCurgus}
it is no restriction to assume that
\begin{enumerate}
\item[(i)] $q\equiv 0$,
\item[(ii)] $\mB=0$ if we consider problem (\ref{eq:sp}), (\ref{eq:c1}), i.e., the operator $L_\mB$,
\item[(iii)] $d=0$ if we consider problem (\ref{eq:sp}), (\ref{eq:c2.3}), i.e., the operator $L_{c,d}$.
\end{enumerate}
Note that the boundary conditions (\ref{eq:c3}), (\ref{eq:c2.1}), (\ref{eq:c2.2}), and (\ref{eq:c1}) with $b_{12}=0$
are separated and the boundary conditions (\ref{eq:c1}) with $b_{12} \neq 0$ and (\ref{eq:c2.3}) are non-separated.

\section{A single turning point and Dirichlet boundary conditions}

In this section we consider the eigenvalue problem \eqref{separated}
(with Dirichlet boundary conditions)
with $a=-1, \; b=1$. We also assume that $r$ has one turning point $x_1 = 0$. By Corollary~\ref{CorCurgus} and Subsection~\ref{sub:conclusion} we can additionally assume that $q = 0$ and $xr(x) > 0$ for a.a. $x \in [-1,1]$ without restriction of generality. Then problem \eqref{separated} takes the form
\begin{equation}\label{separated:-1,1}
-f''  = {\lambda} r f
\quad \mbox{on} \quad [-1,1], \qquad
f(-1) = f(1) = 0.
\end{equation}
In this section we recall some known conditions for the Riesz basis property of \eqref{separated:-1,1} and present some improvements and extensions of these results.

\subsection{Parfenov's Theorem}
We need the following notation
\begin{equation}\label{eq:i_pm_no_x}
I^+(x):=\int_{0}^{x}|r| dt,\ \
I^-(x):=\int_{-x}^{0}|r| dt,\ \
I(x):=I^+(x) + I^-(x)=\int_{-x}^{x}|r| dt
\end{equation}
for $x \in [0,1]$.
The next result is a combination of \cite[Corollary 4, Theorem 6]{P} and \cite[Theorem 3, Corollary 8]{Pa2}.

\begin{theorem}[Parfenov]\label{th:parfenov}
Let $r\in L^1[-1,1]$ and $xr(x)>0$ for a.a. $x \in [-1,1]$.
\begin{enumerate}
 \item[(i)]
Problem \eqref{separated:-1,1} has the Riesz basis property if one of the following conditions is satisfied:

\begin{enumerate}
 \item[(APsuf)]
there exist $C>0$, $\beta>0$ such that for all $0<t\le x\le 1$
\begin{equation}
\min\{I^-(t),I^+(t)\}\le C\left(\frac{t}{x}\right)^\beta I(x); \nonumber
\end{equation}
\item[(API+)]
there exists $t\in (0,1)$ such that
\begin{equation}
I^+(xt)\le \frac{1}{2}I^+(x),\quad x\in (0,1).  \nonumber
 \end{equation}
\end{enumerate}

\item[(ii)] If problem \eqref{separated:-1,1} has the Riesz basis property, then
the following condition is satisfied:
\begin{enumerate}
\item[(APnec)]
there exist $t\in(0, 1]$ and $C\in (0,1/4)$ such that for all
$x\in (0,1)$
\begin{equation}
I^-(xt)I^+(xt)\le C I(xt)I(x). \nonumber
\end{equation}
\end{enumerate}
\item[(iii)] If $r$ is odd, then the conditions {\rm (APsuf)}, {\rm (APnec)}, and {\rm (API+)}
are all equivalent and further equivalent to the Riesz basis property of problem \eqref{separated:-1,1}.
\end{enumerate}
\end{theorem}

\subsection{Positively increasing  functions}

In this subsection we point out a connection between Theorem~\ref{th:parfenov} and the class
of positively increasing functions (see, e.g., \cite{BGT, bks, Rog} and also \cite{K2, K3}).

\begin{definition}\label{def:osv}
Let $\varepsilon > 0$. Assume that $f:(0,\varepsilon)\to (0,\infty)$ is
nondecreasing and $\lim_{x \searrow 0} f(x)=0$. 
The function $f$ is called {\em positively increasing at $0$} (or simply {\em positively increasing})
if there exists a number $t\in (0,1)$ such that
\begin{equation}\label{eq:osv}
S_f(t):=\limsup_{x\to 0}\frac{f(xt)}{f(x)} 
 < 1.
\end{equation}
\end{definition}

\begin{remark}
The class of positively increasing functions
contains as proper subclasses the class of all regularly
varying functions in the sense of Karamata whose index is $\rho>0$
and the class of all rapidly varying functions in the sense of
de Haan whose index is $+\infty$ (cf. \cite{BGT}).
However, it contains no function from the class of slowly varying
functions whose index is $0$ (cf. \cite{BGT}).
\end{remark}

\begin{lemma}\label{lem:osv}
Let $f:(0,1)\to (0,\infty)$ be
nondecreasing and $\lim_{x \searrow 0} f(x)=0$.
Assume additionally that $f\in AC[0,1]$ and $f'$ is positive a.e. on $[0,1]$.
Then the following statements are equivalent:
\begin{enumerate}
\item[(i)] $f$ is positively increasing,
\item[(ii)] there are $C\in (0,1)$ and $t\in (0,1)$ such that 
\begin{equation}\label{eq:p1}
f(xt)\le Cf(x),\quad x\in (0,1),  \nonumber
 \end{equation}
 \item[(iii)] for each $C\in (0,1)$ there is $t\in (0,1)$  such that
\begin{equation}\label{eq:p1B}
f(xt)\le Cf(x),\quad x\in (0,1),   \nonumber
 \end{equation}
 \item[(iv)] there are  $C,\beta>0$ such that for all $t\in (0,1)$
 \begin{equation}\label{eq:p2}
 f(xt)\le C t^\beta f(x),\quad x\in (0,1),   \nonumber
 \end{equation}
\item[(v)] there are no sequences $(a_n), \, (b_n)$ such that $0<a_n<b_n\le 1$ and
\begin{equation} \notag 
\lim_{n\to \infty} \frac{a_n}{b_n}=0,\qquad \lim_{n\to \infty} \frac{f(a_n)}{f(b_n)}=1.
 \end{equation}
\end{enumerate}
\end{lemma}

\begin{proof}
The proof of equivalences (i) $\Leftrightarrow$ (ii) and  (ii) $\Leftrightarrow$ (iv) $\Leftrightarrow$ (v) can be found in \cite[Proposition 3.11]{BF2} and \cite[Theorem 6]{P}, respectively. To complete the proof it remains to note that the implications (iv) $\Rightarrow$ (iii) $\Rightarrow$ (ii) are obvious.
\end{proof}

\begin{remark}
Lemma \ref{lem:osv} remains true without the assumption that $f$ is absolutely continuous
and strictly increasing (cf. \cite[Theorem 3.1]{sP2} and \cite[Lemma~2]{Rog}).
Note that a first application of condition (v) to the Riesz basis property problem was given in \cite{AP}.
\end{remark}

Parfenov's condition (API+) is clearly equivalent
to condition (ii) in Lemma \ref{lem:osv} with $f = I^+$.
This allows us to reformulate Theorem \ref{th:parfenov}
using the concept of positively increasing functions.

\begin{corollary}\label{cor:pi}
Let $r\in L^1[-1,1]$ and $xr(x)>0$ for a.a. $x \in [-1,1]$.
\begin{itemize}
\item[(i)]
If $I^+$ is positively increasing,
then problem \eqref{separated:-1,1} has the Riesz basis property.
\item[(ii)] If $r$ is odd, then problem \eqref{separated:-1,1} has the Riesz basis property if and only if the function $I^+$ is positively increasing.
\end{itemize}
\end{corollary}

\begin{remark}\label{rem:3.2}
\begin{enumerate}
\item[(i)] Corollary~\ref{cor:pi} holds true if $I^+$ is replaced by  $I^-$.
\item[(ii)] It is clear that the property of $I^+$ to be a positively increasing function
depends only on the local behavior of $r$ near $0$.
Therefore, without loss of generality
in Theorem~\ref{th:parfenov} one can
reformulate the conditions (API+) and, if $r$ is odd, also (APsuf), (APnec)
on an arbitrary small subinterval $x\in (0,\varepsilon)$, $0 < \varepsilon \le 1$.
\end{enumerate}
\end{remark}

\subsection{Extensions of Parfenov's Theorem}\label{ss:odddom}

It this subsection we show that Parfenov's result from Theorem~\ref{th:parfenov}(iii) (or Corollary~\ref{cor:pi}(ii))
carries over to classes of non-odd weight functions which allow
certain types of odd-domination (cf. \cite{F5, BF, BF2}).

\subsubsection{A known extension to strongly odd-dominated weights}\label{ss:odddom_known}
For a function $r\in L^1[-1,1]$
we define its even and odd part by
\begin{equation}\label{eq:oddpart}
r^e(x)=\frac{r(x)+r(-x)}{2},\quad r^o(x)=\frac{r(x)-r(-x)}{2},\quad x\in[-1,1].
\end{equation}
The sign condition $xr(x)>0$ implies $|r^e(x)|< r^o(x)$ for a.a. $x\in(0,1)$.
\begin{definition}[\cite{BF, BF2}]\label{def:odd}
A function $r\in L^1[-1,1]$ with $xr(x)>0$ for a.a. $x \in [-1,1]$
is called {\em weakly odd-dominated}
if there exists a function $\rho:[0,1]\to[0,1)$ such that
\begin{equation}\label{eq:odd}
\int_0^x|r^e|dt\le \rho(\varepsilon)\int_0^x r^odt,\quad x\in[0,\varepsilon], \quad \varepsilon \in [0,1].
\end{equation}
If $\rho$ satisfies
\[
 \rho(\varepsilon)=o(1),\quad \varepsilon\to 0,
\]
then $r$ is called {\em odd-dominated}.
If additionally
\[
 \rho(\varepsilon)=o(\varepsilon^{1/2}),\quad \varepsilon\to 0,
\]
then $r$ is called {\em strongly odd-dominated}.
\end{definition}

Clearly, each odd function is strongly odd-dominated.
The next theorem combines results from \cite[Theorem 4.2, Theorem 4.3]{BF} and \cite[Proposition 3.11, Theorem 3.13]{BF2}.
These results are now formulated in terms of positively increasing functions.

\begin{theorem}[\cite{BF,BF2}]\label{thm:odddom}
Let $r\in L^1[-1,1]$ and $xr(x)>0$ for a.a. $x\in [-1,1]$.
\begin{itemize}
\item[(i)] If $r$ is odd-dominated, then the conditions {\rm (API+)}
and  {\rm (APsuf)} are equivalent
and each of them holds true if and only if
the function $I^+$ is positively increasing.
\item[(ii)] If  $r$ is strongly odd-dominated, then
problem \eqref{separated:-1,1} has the Riesz basis property
if and only if the function $I^+$ is positively increasing.
\end{itemize}
\end{theorem}

\subsubsection{A further extension to odd-dominated weights}\label{ss:odddom_unknown}

Next we extend Theorem~\ref{thm:odddom}~(ii) to the case
of odd-dominated functions. First, we observe the following estimate.
\begin{lemma}\label{lem:r_ro}
Put $I^+_o(x) := \int_{0}^{x}r^o \, dt$
and let $I^\pm(x)$ be given as in \eqref{eq:i_pm_no_x}.
Assume that $r$ is weakly odd-dominated
and let $\rho$ be as in Definition~{\rm\ref{def:odd}}.
\begin{itemize}
 \item[(i)] For all $\varepsilon \in [0,1]$ and $x \in [0,\varepsilon]$ we have
\begin{equation}\label{I_Io}
(1-\rho(\varepsilon))I^+_o(x)
\le I^\pm(x)
\le (1+\rho(\varepsilon))I^+_o(x).
\end{equation}
 \item[(ii)] The function $I^+$ satisfies {\rm (API+)} if and only if $I^+_o$ satisfies {\rm (API+)}.
\end{itemize}
\end{lemma}
\begin{proof}
(i) Notice that
$r^o(x) - |r^e(x)| \leq |r(\pm x)| \leq r^o(x) + |r^e(x)|$ for a.a. $x > 0$.
Therefore, estimate \eqref{I_Io} follows from \eqref{eq:odd}
for $I^+(x)$ as well as for  for $I^-(x)$.

(ii) Assume that $I^+$ satisfies (API+).
Then there is a $t \in (0,1)$ such that with $C:=1/2$
\[
I^+(xt)\le CI^+(x),\quad x\in(0,1).
\]
By Lemma \ref{lem:osv} (iii),
we can also choose $C = \frac{1-\rho(1)}{2(1+\rho(1))} \, (\in (0,1))$.
Then \eqref{I_Io} implies
\[
I^+_o(xt) \leq \frac{1}{1-\rho(1)}I^+(xt)
\leq \frac{C}{1-\rho(1)}I^+(x)
\leq C \frac{1+\rho(1)}{1-\rho(1)}I^+_o(x)
= \frac{1}{2}I^+_o(x)
\]
for all $x \in (0,1)$.
This proves (API+) for $I^+_o$.
The converse follows similarly.
\end{proof}

\begin{theorem}\label{lem:odddom}
Let $r\in L^1[-1,1]$ and $xr(x)>0$ for a.a. $x \in [-1,1]$
and assume that $r$ is odd-dominated.
Then the conditions {\rm (API+)},
{\rm (APsuf)} and {\rm (APnec)} are equivalent
and further equivalent to each of the following statements
(which are then also equivalent):
\begin{itemize}
 \item[(i)] Problem \eqref{separated:-1,1} has the Riesz basis property.
\item[(ii)] The function $I^+$ is positively increasing.
\end{itemize}
\end{theorem}
\begin{proof}
By Theorems~\ref{th:parfenov} and~\ref{thm:odddom},
only the implication (APnec)$\Rightarrow$(API+)
remains to be shown.
Assume that $r$ satisfies (APnec) with some $C<1/4$ and $t\in (0, 1]$.
Since $r$ is odd-dominated, for an arbitrary $\varepsilon\in (0,1)$,
there exists $t=t(\varepsilon)\in (0,1)$ such that $\rho(t)<\varepsilon$.
Then Lemma~\ref{lem:r_ro}(i) imlies
\[
I_o^+(xt) \leq \frac{I^\pm(xt)}{1-\varepsilon}, \quad
I(xt) = I^+(xt) + I^-(xt) \leq 2 (1 + \varepsilon) I_o^+(xt),\quad x\in(0,1),
\]
and hence, using (APnec) we obtain
\[
I_o^+(xt)^2
\leq \frac{I^+(xt) I^-(xt)}{(1-\varepsilon)^2}
\leq \frac{C}{(1-\varepsilon)^2} I(xt) I(x)
\leq \tilde{C} I_o^+(xt) I_o^+(x),
\]
for all $x\in(0,1)$. Here
\[
 \tilde{C}(\varepsilon)=4C \frac{(1+\varepsilon)^2}{(1-\varepsilon)^2}.
\]
 Since $4C<1$, we can find a sufficiently small $\varepsilon_0>0$ such that
$\tilde{C}_0=\tilde{C}(\varepsilon_0) 
<1$ and 
\[
I_o^+(xt) \leq \tilde{C}_0 I_o^+(x),\quad x\in (0,1).
\]
By Lemma \ref{lem:osv} and Remark \ref{rem:3.2} (ii),
$r_o$ satisfies  (API+). It remains to apply Lemma \ref{lem:r_ro} (ii).
\end{proof}

Note again that a similar statement holds true for $I^-$ instead of $I^+$.
We shall see below (Example \ref{ex:wod_not_stable})
that a further generalization of Theorem \ref{thm:odddom}
to a weakly odd-dominated setting is not possible.
Now, Theorem \ref{lem:odddom} together with Lemma \ref{lem:r_ro} (ii)
implies the following generalization of \cite[Theorem 4.4]{BF2}
from strongly odd-dominated weights to odd-dominated weights.

\begin{corollary}\label{cor:r_ro}
Assume that $r$ is odd-dominated and $r^o$ is its odd part \eqref{eq:oddpart}.
Then, the  Riesz basis property of \eqref{separated:-1,1}
is equivalent to the  Riesz basis property in $L_{|r^o|}^2[-1,1]$
of the eigenvalue problem
\begin{equation}\label{eq:star}
-f'' = {\lambda} r^o f
\quad \mbox{on} \quad [-1,1], \qquad
f(-1) = f(1) = 0.
\end{equation}
\end{corollary}

\begin{remark}\label{rem:r_ro_1}
Provided that $r$ is weakly odd-dominated, it follows from  Lemma~\ref{lem:r_ro} and Theorem~\ref{th:parfenov} that if \eqref{eq:star} has the Riesz basis property, then \eqref{separated:-1,1} also has
the Riesz basis property. Now, by Corollary~\ref{cor:r_ro}, the converse is true if $r$ is odd-dominated. However, Example~\ref{ex:wod_not_stable} below shows that the converse is no longer true for weakly odd-dominated weights.
\end{remark}

\subsection{Some consequences of Volkmer's sufficient condition}
In this subsection we derive some consequences of Volkmer's sufficient condition for the Riesz basis property.

The next sufficient condition was originally formulated in \cite[Corollary~2.7]{V} for $r \in L^\infty[-1,1]$. For an unbounded generalization we refer to \cite[Theorem 4.2]{BC}.

\begin{theorem}[Volkmer]\label{thm:volkmer}
Let $r\in L^1[-1,1]$ and let also $xr(x)>0$ for a.a. $x \in [-1,1]$.
Assume that for some $t \in \dR\setminus\{0\}$ and $0 < \varepsilon < \min\{1/|t|, 1\}$ the function
\begin{equation} \notag 
 g(x) := \frac{r(x)}{r(tx)},\quad x\in[-\varepsilon,0),
\end{equation}
is continuously differentiable on $[-\varepsilon,0)$ and has a continuously differentiable extension to $[-\varepsilon,0]$ with $g(0) \neq t$.
Then \eqref{separated:-1,1} has the Riesz basis property.
\end{theorem}

A similar condition was also obtained in \cite[Theorem~3.7]{F2} (for weights $r\in L^1[-1,1]$) but with the restriction $t>0$.
Note that the case $t>0$ also follows from Theorem~\ref{th:parfenov}
since in this case the function $I^-$ is positively increasing.
However, Theorem~\ref{thm:volkmer} will be applied below with $t<0$
and this case is not covered by Theorem~\ref{th:parfenov}
(see remarks after Theorem 6 in \cite{P}).

Let $ r \in L^1[0,1]$ be real and positive, i.e. $r>0$ a.e. on $[0,1]$.
Fix $A,B>0$ and let $b:=\min\{1,B^{-1}\}$. Define the function $\tilde{r}: [-b, b]\to \dR$ by
\begin{equation}\label{eq:scaling}
\tilde{r}(x)= \begin{cases}
r(x), &x\in (0,b),\\
-A\, r(-Bx),& x\in (-b,0)
\end{cases}.
\end{equation}
This function satisfies $x \tilde{r}(x) > 0$ for a.a. $x \in [-b,b]$ and it can be regarded as a ``scaling perturbation'' of the odd extension of $r$.  Now, Theorem~\ref{thm:volkmer} can be used to decide
whether the problem
\begin{equation}\label{eq:sp_tilde}
-f''=\lambda \tilde{r} f\quad \text{on} \quad [-b,b], \qquad
f(-b)=f(b)=0,
\end{equation}
has the Riesz basis property in $L^2_{|\tilde{r}|}[-b,b]$.
This, again, is a particular case of \eqref{separated}.

\begin{corollary}\label{th:AneqB}
Let $A,B>0$ satisfy $A\neq B$.
Let $r\in L^1[0,1]$ be positive on $[0,1]$ and let $\tilde{r}$ be given by \eqref{eq:scaling}.
Then problem \eqref{eq:sp_tilde}
has the Riesz basis property in $L^2_{|\tilde r|}(-b,b)$.
\end{corollary}

\begin{proof}
We apply Theorem~\ref{thm:volkmer} with $t=-B$ and $\varepsilon =b$:
\[
g(x) := \frac{\tilde{r}(x)}{\tilde{r}(-Bx)} =  \frac{-A\, r(-Bx)}{r(-Bx)} = -A, \qquad x \in [-b, 0).
\]
Clearly, $g$ satisfies the assumptions of Theorem~\ref{thm:volkmer} and hence the claim follows.
\end{proof}

The following result
shows that Theorem \ref{lem:odddom} and Corollary \ref{cor:r_ro} cannot be extended to a weakly odd-dominated setting. 

\begin{proposition}\label{thm:wod_not_stable}
Let $r^o\in L^1[-1,1]$ be odd.
\begin{itemize}
\item[(i)] If problem \eqref{eq:star} has the Riesz basis property,
then for any even function $r^e\in L^1[-1,1]$ such that the weight $r:=r^o+r^e$ is weakly odd-dominated, the corresponding problem \eqref{separated:-1,1} has the Riesz basis property.
\item[(ii)] If problem \eqref{eq:star} does not have the Riesz basis property,
then there is an even function $r^e\in L^1[-1,1]$ such that the weight $r:=r^o+r^e$ is weakly odd-dominated and the corresponding problem \eqref{separated:-1,1} has the Riesz basis property.
\end{itemize}
\end{proposition}

\begin{proof}
Claim (i) follows from Lemma~\ref{lem:r_ro}(ii) and Theorem~\ref{th:parfenov}.
For (ii) set
\[
r^e:=-\frac{1}{2}r^o(|x|),\quad x\in [-1,1].
\]
Clearly, $r^e$ is even. Furthermore,
\[
r(x)=r^o(x)+r^e(x)=\begin{cases} \frac{1}{2}r^o(x), & x\in (0,1),\\
\frac{3}{2}r^o(x), & x\in (-1,0),
\end{cases},
\]
and it is easy to check that $r$ is weakly odd-dominated
with $\rho(\varepsilon) = 1/2$.
Finally, by Corollary \ref{th:AneqB}
(with $A=3$ and $B=1$),
the corresponding problem \eqref{separated:-1,1} has the Riesz basis property.
\end{proof}

The following example is related to an example in \cite{P}. 

\begin{example}\label{ex:wod_not_stable}
Set
\[
r(x) = \frac{1}{x(1-\log |x|)^{2}}, \quad x\in [-1,0)\cup (0,1].
\]
Note that $I_+$ is not positively increasing.
Then, by Corollary \ref{cor:pi} (ii),
problem (\ref{separated:-1,1}) does not have the Riesz basis property.
Choose any $A>1$ and consider the function $\tilde{r}$
defined by \eqref{eq:scaling} with $B=1$.
It is easy to see that $\tilde{r}$ is weakly odd-dominated.
Indeed, Definition \ref{def:odd} can be checked with
\[
\tilde{r}^o(x)=\frac{1+A}{2}r(x),\quad \tilde{r}^e(x)=\frac{1-A}{2}r(x),\quad \rho(\varepsilon) = \frac{A-1}{A+1}<1.
\]
Now, by Corollary \ref{th:AneqB},
problem \eqref{eq:sp_tilde} (also on $[-1,1]$) has the Riesz basis property,
whereas the corresponding eigenvalue problem with the weight function  $\tilde{r}^o$ does not have the Riesz basis property.
\end{example}


\section{A finite number of turning points and general boundary conditions}

In this section we explore the Riesz basis property
of general eigenvalue problem \eqref{eq:sp},\eqref{eq:gbc}.
We assume that the weight function $r$ has a finite number $n \geq 1$ of turning points, that is, we assume that there exists a polynomial $p$ of degree $n$ whose roots are simple and lie in $(a,b)$ such that $rp > 0$ a.e. on $[a,b]$.  Then the roots $x_1, \ldots, x_n \in (a,b)$ of $p$ are the points at which $r$ changes sign.  For $x \in [a,b]$ and sufficiently small $\epsilon > 0$ we extend the notation introduced in \eqref{eq:i_pm_no_x} as follows:
\begin{equation}\label{eq:i_pm}
I^+_x(\mu):=\int_{x}^{x+\mu}|r(t)| dt,\quad
I^-_x(\mu):=\int_{x-\mu}^{x}|r(t)| dt, \qquad \mu \in [0,\varepsilon].
\end{equation}
Note that for the functions $I^{\pm}$ defined in \eqref{eq:i_pm_no_x} we have $I^{\pm} = I^{\pm}_0$, i.e., $I^\pm=I_x^\pm$ with $x = 0$.

\subsection{Some consequences of Pyatkov's theorem}\label{sec:pyatkov}

We now present a result from \cite{sP2} and some of its consequences
for a class of weight functions which we call locally odd-dominated.

It follows from Remark~\ref{rem:3.2}(ii) and Corollary~\ref{cor:pi}(ii)
that in the case of an odd weight function $r$,
the Riesz basis property of problem \eqref{separated:-1,1}
depends only on a local behavior of the weight function at the turning point. It was shown by S.G. Pyatkov \cite{sP2} that the latter holds true in general. Namely, the result from \cite[Theorem~4.2]{sP2} which we paraphrase below, reduces the Riesz basis property of (\ref{eq:sp}), (\ref{eq:gbc}) to problems on the subintervals $[a_k,b_k]$, $k=1,\ldots,n$, where
\begin{equation}\label{intervals}
a \le a_1 < x_1 < b_1 \le a_2 < x_2 < b_2 \le a_3 < \cdots < b_{n-1} < a_n < x_n < b_n \le b
\end{equation}
and with the Dirichlet boundary conditions. Pyatkov's \cite[Theorem~4.2]{sP2} was originally given in the terminology of Theorem~\ref{th:parfenov}; we reformulate it in terms of positively increasing functions.

\begin{theorem}[Pyatkov]\label{ThmPyatkov}
Let $q,r\in L^1[a,b]$ and $C,D\in\dC^{2\times2}$ satisfy \eqref{CD}. Assume one of the following conditions:
\begin{enumerate}
\item[(SPsep)] the boundary conditions \eqref{eq:gbc} are separated;
\item[(SPeve)] $n$ is even (i.e. $r$ has the same sign on $(a, x_1)$ and $(x_n, b)$);
\item[(SPa)] the function $I_a^+$ given by \eqref{eq:i_pm} with $x=a$ is positively increasing;
\item[(SPb)] the function $I_b^-$ given by \eqref{eq:i_pm} with $x=b$ is positively increasing.
\end{enumerate}
Then, problem \eqref{eq:sp}, \eqref{eq:gbc} has the Riesz basis property if and only if for every $k \in \{1,\ldots, n\}$  the problem
\begin{equation}\label{separatedLoc}
-f'' + q f = {\lambda} r f
\quad \mbox{on} \quad [a_k,b_k], \qquad
f(a_k) = f(b_k) = 0
\end{equation}
has the Riesz basis property in
$L_{|r|}^2[a_k,b_k]$.
\end{theorem}

The following particular case of Theorem~\ref{ThmPyatkov} emphasizes that the Riesz basis property depends only on the local behavior of the weight function near its turning point.

\begin{corollary}\label{ThmPyatkov2}
Let $q,r\in L^1[-1,1]$ and  $xr(x)>0$ for a.a. $x\in [-1,1]$.
Then for arbitrary $a_1,b_1$ such that $-1 < a_1 < 0 <  b_1 < 1$
the Riesz basis property of \eqref{separated:-1,1} in $L_{|r|}^2[-1,1]$
is equivalent to the Riesz basis property of
\begin{equation} \notag 
-f'' + q f = {\lambda} r f
\quad \mbox{on} \quad [a_1,b_1], \qquad
f(a_1) = f(b_1) = 0
\end{equation}
in $L_{|r|}^2[a_1,b_1]$.
\end{corollary}

\begin{remark}\label{Rem_q_unbounded}
(i)
Theorem \ref{ThmPyatkov} was stated in \cite{sP2} for $q\in L^\infty[a,b]$.
However, by Corollary~\ref{CorCurgus}(i), it remains true for an arbitrary $q\in L^1[a,b]$.

(ii)
Problem \eqref{eq:sp} with one turning point
and anti-periodic boundary conditions
(i.e. (\ref{eq:c2.3}) with $c=-1, \, d=0$) was studied in \cite{BC2}.
Neither of the conditions (SPsep), (SPeve) of Theorem~\ref{ThmPyatkov} holds in this case. It was pointed out in \cite{BC2}  that the problem
with Dirichlet boundary conditions may have the Riesz basis property, while the problem with anti-periodic boundary conditions may not have the Riesz basis property. In this case $r$ does not satisfy
neither of the conditions (SPa), (SPb). Our Example~\ref{ex:periodic_2} below offers another similar example with periodic boundary conditions.
\end{remark}

Next, we give a local version of Definition~\ref{def:odd}.

\begin{definition}\label{def:locodd}
A function $r$ is called {\em locally odd at} $x_0 \in (a,b)$ if there exists $\varepsilon > 0$ such that the function $\tilde{r}(x) = r(\varepsilon x+x_0)$ is odd on $[-1,1]$. A function $r$ is called {\em locally odd-dominated at} $x_0$ if there exists $\varepsilon > 0$ such that either $\tilde{r}$ or $-\tilde{r}$ is odd-dominated on $[-1,1]$.
Furthermore, the function $r$ is called {\em locally odd at the boundary} if there exists  $\varepsilon >0$ such that $r(a+x)=-r(b-x)$ for a.a. $x\in (0,\varepsilon)$.
\end{definition}

The first part of the following corollary is a generalization of Theorem~\ref{lem:odddom} to locally odd-dominated weights.
The second part is a generalization of \cite[Lemma 4.3]{sP2}
and \cite[Theorem 2.1]{K} from locally odd to locally odd-dominated weights.

\begin{corollary}\label{ParfenovSufficient}
\begin{enumerate}

\item[(i)]
Let $q,r\in L^1[-1,1]$ and  $xr(x)>0$ for a.a. $x \in [-1,1]$
and assume that $r$ is locally odd-dominated at $0$.
Then \eqref{separated:-1,1} has the Riesz basis property
if and only if the function $I^+$ is positively increasing.

\item[(ii)]
Let $q, r \in L^1[a,b]$ and $C,D\in\dC^{2\times2}$ satisfy \eqref{CD}. Assume that $r$ has $n$ turning points and that it is locally odd-dominated at each turning point $x_k$, $k \in \{1,\ldots, n\}$. Moreover, assume that one of the conditions {\rm (SPsep)}, {\rm (SPeve)}, {\rm (SPa)} and {\rm (SPb)} from Theorem~{\rm\ref{ThmPyatkov}} is satisfied.
Then, problem \eqref{eq:sp}, \eqref{eq:gbc} has the Riesz basis property if and only if for every $k \in \{1,\ldots, n\}$ the functions $I_{x_k}^+$ defined by \eqref{eq:i_pm} are positively increasing.
\end{enumerate}
\end{corollary}

\begin{proof}
To prove (i), assume that the function $\tilde{r}$ with $\varepsilon >0$ from Definition~\ref{def:locodd} is odd-dominated. Put $a_1=-\varepsilon$, $b_1=\varepsilon$. Then the statement follows from Corollary~\ref{ThmPyatkov2} and Theorem~\ref{lem:odddom}
using the isometric transformation $\Phi : L^2_{|r|}[-\varepsilon,\varepsilon] \rightarrow L^2_{|\tilde{r}|}[-1,1]$
with $\Phi(f)(x) := \sqrt{\varepsilon} f(\varepsilon x)$.
Note that the function $I^+$ associated with $r$ is positively increasing if and only if the analogous function associated with $\tilde{r}$ is positively increasing.

Claim (ii) follows similarly using Theorem~\ref{ThmPyatkov} instead of Corollary~\ref{ThmPyatkov2}.
\end{proof}

\subsection{The Riesz basis property for the general eigenvalue problem}\label{ss:4.2}

The main objective of this subsection is to improve Pyatkov's Theorem.
As it was noticed in Remark~\ref{Rem_q_unbounded}(ii),
if conditions (SPsep) and (SPeve) of Theorem~\ref{ThmPyatkov} are not fulfilled, then there are cases when problem \eqref{eq:sp}, \eqref{eq:gbc} does not have the Riesz basis property, even though all the problems in \eqref{separatedLoc} have the Riesz basis property. However, by Theorem~\ref{ThmPyatkov}, one of the conditions (SPa) or (SPb) is sufficient to guarantee that the Riesz basis property of each of the problems in \eqref{separatedLoc}  implies the Riesz basis property for problem \eqref{eq:sp}, \eqref{eq:gbc}. In this subsection we shall show in which cases of nonseparated boundary conditions (SPa) or  (SPb) is either obsolete or necessary.

In view of Subsections~\ref{operators} and~\ref{sub:conclusion},
nonseparated boundary conditions may only be of two types: either \eqref{eq:c1} or \eqref{eq:c2.3}. We begin with the case of boundary conditions \eqref{eq:c1}.

\begin{proposition}\label{ParfenovNecessary_2}
Let $q, r \in L^1[a,b]$ and assume that $r$ has $n$ turning points $x_1,\ldots,x_n$. Let  $a_1,\ldots,a_n$ and $b_1,\ldots,b_n$ satisfy  \eqref{intervals}.  Then:
\begin{enumerate}
\item[(i)] Problem \eqref{eq:sp}, \eqref{eq:c1} has the Riesz basis property
if and only if for each $k \in \{1,\ldots, n\}$ the problem in
\eqref{separatedLoc} has the Riesz basis property in
$L_{|r|}^2[a_k,b_k]$.
\item[(ii)]
Assume that $r$ is locally odd-dominated at each turning point $x_k$, $k \in \{1,\ldots, n\}$. Then problem
\eqref{eq:sp}, \eqref{eq:c1} has the Riesz basis property
if and only if for each $k \in \{1,\ldots, n\}$ the function $I_{x_k}^+$ is positively increasing.
\end{enumerate}
\end{proposition}

\begin{proof}
By Corollary~\ref{CorCurgus} and Subsection~\ref{sub:conclusion}
we can choose the matrix $\mB$ in \eqref{eq:c1} to be $0$. Then \eqref{eq:c1} are the Neumann boundary conditions and hence separated.
Now, the statement follows from Theorem~\ref{ThmPyatkov} and
Corollary~\ref{ParfenovSufficient} by condition (SPsep).
\end{proof}

We continue with the eigenvalue problem \eqref{eq:sp} with
the boundary conditions \eqref{eq:c2.3}:
\begin{equation}\label{eq:evp15}
-f'' + q f = {\lambda} r f
\quad \mbox{on}
\quad [a,b], \qquad
 c f(a) = f(b), \quad f'(a) =  \overline{c} f'(b)+df(a)
\end{equation}
with $c\in \dC\setminus\{0\}$ and $d\in\dR$.
Together with problem \eqref{eq:evp15} we will consider the following auxiliary eigenvalue problem
\begin{equation}\label{eq:evptilde}
-f'' = {\lambda} \widetilde{r} f
\quad \mbox{on}
\quad [\widetilde{a},\widetilde{b}], \qquad
c f(\widetilde{a}) = f(\widetilde{b}), \quad f'(\widetilde{a}) = \overline{c}  f'(\widetilde{b}),
\end{equation}
where
\begin{equation}\label{eq:ti_a}
\widetilde{a}:=a - \frac{\varepsilon}{|c|^2},\quad
\widetilde{b}:=b - \varepsilon
\end{equation}
with  $\varepsilon$ such that $0 < \varepsilon < \frac{1}{2} \min\{x_1 - a, b - x_n\}$ and where the weight function $\widetilde{r}$ is given by
\begin{equation}\label{eq:ti_r}
\widetilde{r}(x):=
\begin{cases}
|c|^4r(b - |c|^2(a - x)),& x\in [\widetilde{a},a)  \\
r(x), & x\in [a,\widetilde{b}]
\end{cases}.
\end{equation}
Clearly, the weight function $\widetilde{r}$ is obtained from $r$ by shifting and scaling.
Notice that $x_k$ is a turning point for $\widetilde{r}$ for each $k\in\{1,\dots,n\}$. Moreover, if $n$ is odd, then $x_0 := a$ is an additional turning point of $\tilde{r}$. Hence, if $n$ is odd, the weight function $\tilde{r}$ in problem \eqref{eq:evptilde} has the even number $n+1$ of turning points $x_0,x_1,\ldots, x_n$. Moreover, the next result shows that the spectral properties of problems \eqref{eq:evp15} and \eqref{eq:evptilde} are closely connected.

\begin{lemma}\label{lem:shift}
Let $q, r \in L^1[a,b]$ and assume that $r$ has $n$ turning points.
Problem \eqref{eq:evp15} has the Riesz basis property in $L^2_{|r|}[a,b]$
if and only if problem \eqref{eq:evptilde} has the Riesz basis property in $L^2_{|\widetilde{r}|}[\widetilde{a},\widetilde{b}]$.
\end{lemma}

\begin{proof}
First, notice that by Corollary \ref{CorCurgus} it suffices to prove the claim for $q\equiv 0$ and $d=0$.
In this case problem \eqref{eq:evp15} has the form
\begin{equation} \notag 
-f'' = {\lambda} r f
\quad \mbox{on}
\quad [a,b], \qquad
c f(a) = f(b), \quad f'(a) = \overline{c}  f'(b)
\end{equation}
which according to Section~\ref{operators} is represented by the operator $L_{c,0}$
in $L^2_{|r|}[a,b]$ (i.e. $L_{c,d}$ with $d=0$).
Let $\widetilde{L}_{c,0}$ denote the corresponding operator
associated with the shifted eigenvalue problem \eqref{eq:evptilde}
in $L^2_{|\widetilde{r}|}[\widetilde{a},\widetilde{b}]$.
It will be shown that these operators are unitarily equivalent
by means of an Hilbert space isomorphism
$\Phi: L^2_{|r|}[a,b]\rightarrow L^2_{|\widetilde{r}|}[\widetilde{a},\widetilde{b}]$.
To this end, for $f \in L^2_{|r|}[a,b]$ consider the function
\[ 
\Phi (f)(x) :=\begin{cases}
                               \frac{1}{c}\ f(b - |c|^2(a - x)),& x \in [\widetilde{a}, a],\\
                               f(x), & x\in (a,\widetilde{b}] .
                               \end{cases}
\]
Clearly, $\Phi (f) \in L^2_{|\widetilde{r}|}[\widetilde{a},\widetilde{b}]$ and, moreover, we have
\begin{eqnarray*}
\int_{\widetilde{a}}^{\widetilde{b}} |\Phi (f)|^2 |\widetilde{r}| \; dx
&=& \int_{\widetilde{a}}^{a} |c|^2 |f(b - |c|^2(a - x))|^2 |r(b - |c|^2(a - x))| \; dx \\&+& \int_a^{\widetilde{b}} |f|^2 |r| \; dx  
= \int_a^{b} |f|^2 |r| \; dx .
\end{eqnarray*}
It is straightforward to check that for $g \in L^2_{|\widetilde{r}|}[\widetilde{a},\widetilde{b}]$ the inverse map is given by
\[
\Phi^{-1}(g)(x)=\begin{cases}
                               g(x), & x\in [a,\widetilde{b}] ,\\
                               c\ g\bigl(a - \frac{1}{|c|^2}(b - x)\bigr),& x \in (\widetilde{b}, b].
                               \end{cases}
\]
Therefore, $\Phi: L^2_{|r|}[a,b]\rightarrow L^2_{|\widetilde{r}|}[\widetilde{a},\widetilde{b}]$
is isometric and one-to-one.
The next step is to show $\Phi (\dom(L_{c,0})) = \dom(\widetilde{L}_{c,0})$.
Indeed, if $f\in \dom(L_{c,0})$, then the boundary conditions imply
\begin{eqnarray*}
&\Phi (f)(a-)=\frac{1}{c}f(b)=f(a)=\Phi (a +),\\
&\Phi (f)'(a-)=\frac{|c|^2}{c}f'(b)=\overline{c}f'(b)=f'(a)=\Phi (f)'(a+).
\end{eqnarray*}
Hence $\Phi (f)$ and $\Phi (f)'$ are absolutely continuous on $[\widetilde{a},\widetilde{b}]$.
Furthermore, we have
\[
\frac{-1}{\widetilde{r}(x)}\Phi (f)''(x)
= \frac{-f''(b - |c|^2(a - x))}{c \, r(b - |c|^2(a - x))}
= -\Phi\Bigl(\frac{1}{r}f''\Bigr)(x) \quad \mbox{ for a.a. } x \in (\widetilde{a}, a)
\] 
and $\Phi (f)$ satisfies the boundary conditions from \eqref{eq:evptilde}:
\begin{eqnarray*}
&c \Phi (f)(\widetilde{a}) = f(b - |c|^2(a - \widetilde{a})) = f(b - \varepsilon) = \Phi (f)(\widetilde{b}),\\
&\Phi (f)'(\widetilde{a})=\overline{c} f'(b - |c|^2(a - \widetilde{a})) =\overline{c} f'(b - \varepsilon)
=\overline{c}\Phi (f)'(\widetilde{b}).
\end{eqnarray*}
Therefore, $\Phi (f) \in \dom(\widetilde{L}_{c,0})$ and $\widetilde{L}_{c,0}(\Phi (f)) = \Phi (L_{c,0} (f))$.
Similarly, for any $g \in \dom(\widetilde{L}_{c,0})$
we can show that  $\Phi^{-1} (g) \in \dom(L_{c,0})$
and moreover $L_{c,0} (\Phi^{-1} (g))=\Phi^{-1} (\widetilde{L}_{c,0}(g))$.
This implies $L_{c,0} = \Phi^{-1}\widetilde{L}_{c,0}\Phi$.
Consequently, $L_{c,0}$ and $\widetilde{L}_{c,0}$ are unitarily equivalent. This completes the proof since a Hilbert space isomorphism preserves the Riesz basis property.
\end{proof}

Combining Theorem~\ref{ThmPyatkov}
with Lemma~\ref{lem:shift}, we obtain the following result.

\begin{lemma}\label{lem:4.7}
Let $q, r \in L^1[a,b]$ and assume that $r$ has $n$ turning points $x_1,\ldots,x_n$. Assume also that $n$ is odd and $a_1, \ldots, a_n$ and $b_1, \ldots, b_n$ satisfy  \eqref{intervals}.
Let  $\tilde{a}$, $\tilde{b}$ and $\widetilde{r}$ be defined by \eqref{eq:ti_a}, \eqref{eq:ti_r} and such that $b_n < \tilde{b}$. Additionally, let  $a_0, b_0$ be such that $\tilde{a} < a_0 < a < b_0 < a_1$.  Then problem \eqref{eq:evp15} has the Riesz basis property in $L^2_{|r|}[a,b]$ if and only if both of the following conditions are satisfied:
\begin{enumerate}
\item[(i)]
for every $k \in \{1,\ldots, n\}$, the problem in \eqref{separatedLoc} has the Riesz basis property in $L_{|r|}^2[a_k,b_k]$,
\item[(ii)]
the problem
\begin{equation}\label{eq:shiftLoc}
-f''=\lambda \widetilde{r}\, f\quad \text{on}\quad [a_0,b_0],\qquad f(a_0)=f(b_0)=0,
\end{equation}
has the Riesz basis property in $L_{|\tilde r|}^2[a_0,b_0]$.
\end{enumerate}
\end{lemma}

\begin{proof}
By Lemma~\ref{lem:shift}, problem \eqref{eq:evp15} has the Riesz basis property if and only if problem \eqref{eq:evptilde}
has the Riesz basis property. Applying Theorem~\ref{ThmPyatkov} to \eqref{eq:evptilde} (with condition (SPeve))
and using \eqref{eq:ti_r},
completes the proof.
\end{proof}

Next, we consider a locally odd-dominated weight $r$.
Under the assumptions that $n$ is odd and $r$ is locally odd at the boundary we can now show that conditions (SPa) and (SPb)
are also necessary if (and only if) $c = \e^{\I t}$, with some $t \in [0, 2\pi)$, in which case the eigenvalue problem \eqref{eq:evp15} is of the form \eqref{nonseparated}.

\begin{proposition}\label{ParfenovNecessary}
Let $q, r \in L^1[a,b]$ and assume that $r$ has $n$ turning points $x_1,\ldots,x_n$.
Assume that $n$ is odd and that
$r$ is locally odd at the boundary
and locally odd-dominated at each turning point $x_k$, $k \in \{1,\ldots, n\}$.
Let $c\in \dC\setminus\{0\}$ and $d\in\dR$.
\begin{enumerate}
\item[(i)] If $|c| \neq 1$, then problem \eqref{eq:evp15}
has the Riesz basis property
if and only if for each $k \in \{1,\ldots, n\}$ the function $I_{x_k}^+$ is positively increasing.
\item[(ii)] If $|c| = 1$, then problem \eqref{eq:evp15}
(i.e. \eqref{nonseparated})
has the Riesz basis property
if and only if the functions $I_a^+$ and $I_{x_k}^+$, $k \in \{1,\ldots, n\}$, are positively increasing.
\end{enumerate}
\end{proposition}

\begin{proof}
By Lemma~\ref{lem:4.7}, the Riesz basis property for \eqref{eq:evp15} is equivalent
to the Riesz basis property of problems \eqref{separatedLoc} and \eqref{eq:shiftLoc}.
Since $r$ is locally odd-dominated at $x_k$, by Corollary \ref{ParfenovSufficient},
problems \eqref{separatedLoc} have the Riesz basis property if and only if the functions $I_{x_k}^+$ are positively increasing.

Further, notice that the weight function $\widetilde{r}$ given by \eqref{eq:ti_r}
has the form \eqref{eq:scaling} in a neighborhood of $x=a$
since $r$ is locally odd at the boundary.
If $|c|=1$, then $\widetilde{r}$ is locally odd at $x=a$
and hence \eqref{eq:shiftLoc} has the Riesz basis property
if and only if $I_a^+$ is positively increasing.
This proves (ii).
If $|c|\neq 1$, then by Corollary \ref{th:AneqB}
(isometrically transformed to a neighbourhood of $a$),
problem \eqref{eq:shiftLoc} always has the Riesz basis property.
This completes the proof.
\end{proof}

Finally, we address the case of general self-adjoint boundary conditions \eqref{eq:gbc} satisfying \eqref{CD}. In Subsection~\ref{operators} we showed that the general self-adjoint boundary conditions
can be rewritten as one of the conditions
\eqref{eq:c1},  \eqref{eq:c3}, \eqref{eq:c2.1}, \eqref{eq:c2.2}, or  \eqref{eq:c2.3}.
Combining this classification of self-adjoint boundary conditions
and Corollary~\ref{ParfenovSufficient} (with condition (SPsep)),
Proposition~\ref{ParfenovNecessary_2},  and Proposition~\ref{ParfenovNecessary},
we arrive at the main result of the paper.
Note that here the case of an even number $n$ of turning points is excluded because this is already covered by Corollary~\ref{ParfenovSufficient} (condition (SPeve)).

\begin{theorem}\label{ThnMain}
Let $q, r \in L^1[a,b]$ and assume that $r$ has $n$ turning points $x_1,\ldots,x_n$. Let $n$ be odd and assume that that $r$ is locally odd-dominated at each turning point $x_k$, $k \in \{1,\ldots, n\}$,
and locally odd at the boundary.
\begin{enumerate}
\item[(i)] If the boundary conditions \eqref{eq:gbc} can be rewritten as
\eqref{eq:c1},  \eqref{eq:c3}, \eqref{eq:c2.1}, \eqref{eq:c2.2}, or \eqref{eq:c2.3}
with $|c| \neq 1$, then problem \eqref{eq:sp}, \eqref{eq:gbc}
has the Riesz basis property
if and only if for all $k \in \{1,\ldots, n\}$ the functions $I_{x_k}^+$ defined by \eqref{eq:i_pm} are positively increasing.
\item[(ii)] If the boundary conditions \eqref{eq:gbc} can be rewritten as
\eqref{eq:c2.3} with $|c| = 1$ (i.e. as in \eqref{nonseparated})
then problem \eqref{eq:sp}, \eqref{eq:gbc}
has the Riesz basis property
if and only if the functions $I_a^+$
(or equivalently $I_b^-$)
and $I_{x_k}^+$ for all $k \in \{1,\ldots, n\}$ are positively increasing.
\end{enumerate}
\end{theorem}

\begin{corollary}\label{cor:similar}
Let $q, r \in L^1[a,b]$ and assume that $r$ has $n$ turning points.
\begin{enumerate}
\item[(i)]
Consider the family of all eigenvalue problems consisting of the equation  \eqref{eq:sp} with an arbitrary $q\in L^1[a,b]$ and subject to any of the boundary conditions in \eqref{eq:c1}, \eqref{eq:c3}, \eqref{eq:c2.1}, \eqref{eq:c2.2}, and \eqref{eq:c2.3} with $|c| \neq 1$.  Then either every such eigenvalue problem has the Riesz basis property or neither of them has the Riesz basis property.
\item[(ii)]
Consider the family of eigenvalue problems \eqref{nonseparated} with an arbitrary potential $q\in L^1[a,b]$ and arbitrary $t \in [0,2\pi)$ and $d \in \dR$. Then either every such eigenvalue problem has the Riesz basis property or neither of them has the Riesz basis property.
\end{enumerate}
\end{corollary}

The following example concludes this section.

\begin{example}\label{ex:periodic_2} 
We consider an eigenvalue problem with periodic boundary conditions:  \begin{equation}\label{eq:periodic}
-f'' = {\lambda} r f
\quad \mbox{on} \quad [-1,1], \qquad
f(-1) = f(1), \quad f'(-1) = f'(1).
\end{equation}
We additionally assume that $r$ is odd, $r(x)=-r(-x)$ for a.a.  $x\in (-1,1)$.
Notice that this is the eigenvalue problem \eqref{nonseparated} with
$a= -1$, $b = 1$, $t=0$, $d=0$, $q = 0$, or \eqref{eq:c2.3} with $c=1, d=0$.

Clearly, the constant function $f_0 = 1$ is an eigenfunction of \eqref{eq:periodic} which corresponds to the eigenvalue $0$. Observe also that the eigenvalue $0$ has a Jordan chain since $\int_{-1}^1 r\, dx=0$. Indeed, in this case the function
\[
g_0(x) := \gamma(x+1) -  \int_{-1}^x (x-t)\; r(t) \, dt \quad \mbox{with} \quad
\gamma:= -\frac{1}{2}\int_{-1}^1t\, r(t) \;  dt
\]
satisfies the boundary conditions
\[
g_0(1) = 0 = g_0(-1), \quad g_0'(1) = \gamma = g_0'(-1)
\]
and $\ell[g_0] = -\frac{1}{r}g_0'' = 1 = f_0$.
Thus $g_0$ is a root function of the eigenvalue problem \eqref{eq:periodic}.

By Theorem~\ref{ThnMain} (or Proposition~\ref{ParfenovNecessary})
the eigenvalue problem \eqref{eq:periodic} has the Riesz basis property
if and only if both functions $I_0^+$ and $I_1^-$
(i.e., $I_x^\pm$ with $x=0,1$)
are positively increasing.
We illustrate this with two specific examples of the weight function $r$.

{\rm (i)} \ We
start with the simplest example: $r(x) = \sgn(x)$, $x \in [-1,1]$.  Then $I_0^+(x) =I_1^-(x)= x$.  Clearly, these functions are positively increasing (moreover, they are regularly varying functions at $0$ in the sense of Karamata with index $\rho=1$, see \cite{BGT}). Therefore, the eigenvalue problem \eqref{eq:periodic} has the Riesz basis property in this case.
Note that the root functions $f_0$ and $g_0$ associated with the eigenvalue 0 are
\[
f_0(x) = 1, \qquad g_0(x)=
\frac{x(1-|x|)}{2}.
\]

{\rm (ii)} \ Consider now the following example:
\begin{equation}\label{eq:r=log}
r(x)=\frac{\sgn(x)}{(1-|x|)\log^2(\frac{1-|x|}{\e})},
\quad  x \in (-1,1) .
\end{equation}
Then $r$ is odd, $r(0+)=1$ and hence for all $t \in (0,1)$
by l'H\^opital's rule we have
\[
\lim_{x\searrow 0}\frac{I_0^+(tx)}{I_0^+(x)}
= \lim_{x\searrow 0}\frac{t\, r(tx)}{r(x)} = t\,\frac{r(0+)}{r(0+)}=t.
\]
Therefore, the function $I_0^+$ is positively increasing.
On the other hand,
\[
I_1^-(x)=\int_{1-x}^1 \frac{1}{(1-t)\log^2(\frac{1-t}{\e})}dt=\frac{1}{1-\log(x)}
\]
and hence we have
\[
\lim_{x\searrow 0}\frac{I_1^-(xt)}{I_1^-(x)}=\lim_{x\searrow  0}\frac{1-\log(x)}{1-\log(x) - \log(t)}=1,
\]
for all $t \in (0,1)$.
Consequently, $I_1^-$ is a slowly varying function in the sense of Karamata \cite{BGT} and hence not positively increasing.
Thus, the periodic eigenvalue problem \eqref{eq:periodic} with $r$ given in \eqref{eq:r=log} does not have the Riesz basis property.

However, the eigenvalue problem \eqref{separated:-1,1}, with Dirichlet boundary conditions and $r$ from \eqref{eq:r=log}, does have the Riesz basis property by Theorem~\ref{th:parfenov}.
Similarly, by Proposition~\ref{ParfenovNecessary}, for any $c\in \mathbb{C}$ with $|c|\neq 1$, the Riesz basis property also holds for the following problem
\[
-f'' = {\lambda} r f
\quad \mbox{on} \quad [-1,1], \quad
cf(-1) = f(1), \quad f'(-1) = \overline{c} f'(1).
\]
\end{example}

\section{Periodic boundary conditions and the regular HELP-type inequality}

Consider the regular HELP-type inequality \eqref{Everitt}
with $b > 0$ and $q, r \in L^1[0,b]$ such that $r > 0$ a.e. on $[0,b]$.
We say that {\it inequality \eqref{Everitt} is valid} if
there exists a constant $K > 0$ such that
(\ref{Everitt}) holds true for all functions
$f\in AC[0,b]$
for which $\frac{1}{r}f' \in AC[0,b]$
and
$\ell[f]:=-(\frac{1}{r} f')' + q f \in L^2[0,b]$.

\subsection{Bennewitz' theorem}

In \cite{Bz} Bennewitz proved the following result (in a more general setting and with a different terminology) on the validity of inequality \eqref{Everitt}.

\begin{theorem}[Bennewitz]\label{ThmBennewitz}
Let $q, r \in L^1[0,b]$ be such that $r > 0$ a.e.
Then, inequality \eqref{Everitt} is valid
if and only if both of the following conditions are satisfied:
\begin{enumerate}
\item[(i)] The functions $I_0^+$ and $I_b^-$ given by \eqref{eq:i_pm} are positively increasing.
\item[(ii)] For all solutions $f$ of $\ell[ f] = 0$ we have  $(\frac{1}{r}f')(b)\overline{f(b)} - (\frac{1}{r}f')(0)\overline{f(0)} = 0$.
\end{enumerate}
\end{theorem}

\begin{example}\label{ExBennewitz}
Bennewitz observed in \cite{Bz1} that inequality (\ref{Everitt}) is valid
for $r=1, \, q=-1$ and $b=m\pi$ with $m \in \dN$. Indeed, in this case we have
$I_0^+(x)=I_b^-(x)
= x$ 
and the solutions $\sin(x)$ and $\cos(x)$ of $f'' + f = 0$ clearly satisfy condition (ii).
\end{example}

\subsection{On a connection with the Riesz basis property}\label{sec:bennezitz}

In \cite{EE1, EE2, Et}
Evans and Everitt studied inequality (\ref{Everitt})
for a smaller class of functions restricted by an additional condition
at the boundary $b$, for example $(f'/r)(b) = 0$.
For the case $q=0, b=1$ it was shown in \cite[Theorem 3.13]{BF2}
that inequality (\ref{Everitt})
with the restriction $(f'/r)(1) = 0$
is equivalent to the Riesz basis property of problem
\eqref{separated:-1,1}
with an odd extension of $r$ to $[-1,1]$ (see also \cite{K2} and \cite{K3}, where this result was extended to a wider setting by using a different approach).

Next we establish an analogous result for inequality \eqref{Everitt} in Bennewitz' sense, i.e. without any restriction at $b$. To this end take again the odd extension of $r$ to $[-b,b]$ and an arbitrary extension $q \in L^1[-b,b]$ of $q$.  Then $r$ has a single turning point at $x_1 = 0$ and, as in Example~\ref{ex:periodic_2}, the Riesz basis property of
the periodic problem
\begin{equation}\label{eq:periodic_b}
-f'' = {\lambda} r f
\quad \mbox{on} \quad [-b,b], \qquad
f(-b) = f(b), \quad f'(-b) = f'(b)
\end{equation}
(i.e. \eqref{eq:c2.3} with $c=1, d=0$)
is equivalent by Theorem \ref{ThnMain} to condition (i) of Theorem~\ref{ThmBennewitz}.
Therefore, Theorem~\ref{ThmBennewitz} implies the following result.

\begin{corollary}\label{BennewitzNonseparated}
Let $q, r \in L^1[-b,b]$ with $x r(x) > 0$ for a.a. $x \in [-b,b]$ and assume that $r$ is odd.
Then inequality \eqref{Everitt} is valid
if and only if condition {\rm (ii)} of Theorem~{\rm\ref{ThmBennewitz}}
is satisfied and problem
\eqref{eq:periodic_b}
has the Riesz basis property
in $L^2_{|r|}[-b,b]$.
\end{corollary}

\begin{remark}\label{6replace17}
According to Corollary~\ref{cor:similar}
the eigenvalue problem \eqref{eq:periodic_b}
in Corollary~\ref{BennewitzNonseparated}
can be replaced by any other eigenvalue problem
of the form \eqref{nonseparated} (e.g. with antiperiodic boundary conditions).
\end{remark}

\begin{example}\label{ExBennewitz2}
It follows from the observation in Example~\ref{ExBennewitz} and Corollary~\ref{BennewitzNonseparated}
that the eigenvalue problem
\[
-f'' - f = {\lambda} \, \sgn(x) f, 
\]
subject to boundary conditions 
\[
 \e^{\I t} f(-m\pi) = f(m\pi), \quad f'(m\pi) =  \e^{-\I t}  f'(m\pi) + d \, f(-m\pi)
\]
with $t \in [0, 2\pi)$ and $d \in \dR$ has the Riesz basis property. On the other hand, this fact also follows from considerations in Example~\ref{ex:periodic_2}(i).
\end{example}

\begin{lemma}\label{q=0}
Let $r \in L^1[0,b]$ be such that $r > 0$ a.e. If $q=0$, then condition {\rm (ii)} of Theorem~{\rm\ref{ThmBennewitz}} and hence inequality \eqref{Everitt} cannot be valid.
\end{lemma}

\begin{proof}
The general solution of $(\frac{1}{r} f')'=0$ is given by
$
f(x)=c_1+c_2\int_0^xr \; dt
$
with $c_1,c_2\in\dC$.
Therefore,
$
(\frac{1}{r}f')(b)\overline{f(b)} - (\frac{1}{r}f')(0)\overline{f(0)}
=|c_2|^2\int_{0}^br \, dt
$
which cannot vanish for $c_2 \neq 0$.
\end{proof}

\begin{remark}
In \cite[Corollary~3.15]{BF2}
it was shown that inequality \eqref{Everitt} (with $b = 1$) is valid
if and only if $I^-_1$ is positively increasing,
condition (ii) of Theorem~\ref{ThmBennewitz} is valid
and problem \eqref{separated:-1,1} has the Riesz basis property.
However, in \cite{BF2} no potential was considered (i.e. $q=0$)
and Lemma~\ref{q=0} was overlooked.
Now by Corollary~\ref{CorCurgus} one can immediately add a potential $q\in L^1[a,b]$.
\end{remark}
\begin{example}
Let $r\in L^1[0,1]$ be positive.
Put $q := -r$ and consider the equation
\[
-\Big(\frac{1}{r}f'\Big)'-r f=0\quad \text{on} \quad [0,1].
\]
Then, a fundamental system of solutions of this equation is
\[
f_1(x)=\sin(I^+(x)),\quad f_2(x)=\cos(I^+(x)),\quad \text{where}\quad I^+(x)=\int_0^x r \, dt.
\]
Moreover, one easily gets
\[
(\tfrac{1}{r}f_j')(1)\overline{f_j(1)} - (\tfrac{1}{r}f_j')(0)\overline{f_j(0)}=
\frac{(-1)^{j+1}}{2}\sin(2I^+(1)), \quad j\in\{1,2\}.
\]
Now, as in Example \ref{ex:periodic_2}(ii), choose
\[
r(x)=\frac{\pi}{(1-x)\log^2(\frac{1-x}{\e})},\quad x\in(0,1)
.\]
 Then we compute $R(1)=\int_0^1 r \, dt =\pi$ and hence
 we arrive at an example such that condition (ii) of Theorem \ref{ThmBennewitz} is
satisfied, the function $I_0^+$ is positively increasing, however the function $I_1^-$ is not positively increasing.
This implies that inequality \eqref{Everitt} (with $q=-r$)
fails to hold.
On the other hand, in addition to condition (ii) of Theorem~\ref{ThmBennewitz},
inequality \eqref{Everitt} with $q = 0$ and
restricted to the class of functions
such that $(f'/r)(1) = 0$ is also valid
by \cite[Theorem 3.13]{BF2}.
\end{example}

\end{document}